\documentclass[11pt,reqno]{amsart}
\usepackage[foot]{amsaddr}

\numberwithin{equation}{section}
\allowdisplaybreaks

\usepackage{setspace}

\usepackage{parskip}
\usepackage{fullpage}
\usepackage{amssymb,amsfonts,amsmath,amsthm}
\usepackage{verbatim}

\usepackage{hyperref}
\usepackage[dvipsnames]{xcolor}

\newcommand\myshade{85}
\colorlet{mylinkcolor}{violet}
\colorlet{mycitecolor}{YellowOrange}
\colorlet{myurlcolor}{Aquamarine}
\hypersetup{
	linkcolor  = mylinkcolor!\myshade!black,
	citecolor  = mycitecolor!\myshade!black,
	urlcolor   = myurlcolor!\myshade!black,
	colorlinks = true,
}

\usepackage[shortlabels]{enumitem}
\newlist{sumside}{itemize}{2}
\setlist[sumside]{topsep=0pt,label=-,leftmargin=2em,noitemsep}

\usepackage{etoolbox}
\patchcmd{\section}{\scshape}{\bfseries}{}{}
\makeatletter
\renewcommand{\@secnumfont}{\bfseries}
\makeatother

\makeatletter
\patchcmd{\@settitle}{\uppercasenonmath\@title}{\Large}{}{}
\patchcmd{\@setauthors}{\MakeUppercase}{\normalsize}{}{}
\makeatother

\makeatletter
\def\th@definition{%
	\thm@headfont{\bfseries}
	\normalfont 
}
\makeatother

\usepackage[compress]{cite}

\usepackage{newtxtext}
\usepackage[libertine]{newtxmath}

\usepackage{stmaryrd}

\theoremstyle{definition}

\numberwithin{define}{section}
\newtheorem{thm}{Theorem}
\numberwithin{thm}{section}

\newtheorem{cor}[thm]{Corollary}
\newtheorem{rmk}[thm]{Remark}

\newtheorem*{exam-non}{Example}

\makeatletter
\def\th@definition{
	\thm@headfont{\bfseries}
	\normalfont 
}
\makeatother

\newcommand{\ignore}[1]{}

\newcommand{\nMod}[2]{
	\not\equiv #1\,\,(\text{mod}\,\,#2)
}
\newcommand{\Mod}[2]{
	\,\equiv\, #1\,\,(\text{mod}\,\,#2)
}

\newcommand{\forbid}[4]{[#1,#2; \,\equiv{#3}\,(#4)]}

\newcommand{\cC}{\mathcal{C}}

\newcounter{idfamily}
\newcommand{\newfamily}{
	\refstepcounter{idfamily}
	\refstepcounter{subsection}
	\subsection*{$\blacktriangleright$ Family \theidfamily.}\,
}

\newcounter{familyexample}[idfamily]
\newcommand{\newfamilyexample}[1]{
	\refstepcounter{familyexample}
	\refstepcounter{subsubsection}
	\subsubsection*{$\triangleright$ 
		#1 \theidfamily.\thefamilyexample}\quad
}

\newcommand{\desc}[1]{{\textit{#1:}}}

\usepackage{xstring,refcount}
\newcommand{\refid}[1]{\StrBehind{\getrefnumber{#1}}{.}}

\begin{document}

\title[New identities of Rogers-Ramanujan-MacMahon type]{A variant of {\bf \tt IdentityFinder} and some new identities of Rogers-Ramanujan-MacMahon type}

\author[Kanade]{Shashank Kanade\textsuperscript{1}}
\address{\textsuperscript{1}University of Denver, Denver, USA}
\email{shashank.kanade@du.edu}
\thanks{S.K. is presently supported by a start-up grant provided by University of Denver}

\author[Nandi]{Debajyoti Nandi}
\email{debajyoti.nandi@gmail.com}

\author[Russell]{Matthew C.\ Russell\textsuperscript{2}}
\address{\textsuperscript{2}Rutgers - The State University of New Jersey, Piscataway, USA}
\email{russell2@math.rutgers.edu}

\dedicatory{To George E.\ Andrews, with great respect and gratitude}
\begin{abstract}
	We report on findings of a variant of \texttt{IdentityFinder} -- a
	Maple program that was used by two of the authors to conjecture several new 
	identities of Rogers-Ramanujan kind. 
	In the present search, we modify the parametrization of the search space
	by taking into consideration several aspects of Lepowsky and Wilson's $Z$-algebraic
	mechanism and its variant by Meurman and Primc. We search for identities
	based on forbidding the appearance of ``flat'' partitions as sub-partitions.
	Several new identities of Rogers-Ramanujan-MacMahon type are found and proved.
\end{abstract}



\maketitle

\section{Introduction and motivation}

The aim of this paper is to report on several new 
integer partition identities of Rogers-Ramanujan-MacMahon type.
This paper should be viewed largely as a continuation of \cite{KR-idfind} with a 
search space that is broader and differently parametrized than the one in \cite{KR-idfind}.
For the general background on experimentally
finding new partition identities and the importance thereof we
refer the reader to a small but by no means exhaustive selection of articles:
\cite{And-Computers,And-qSeriesBook, MSZ, KR-idfind, MR-thesis}.
We ask the reader to recall the relevant terminology (sum-sides, product-sides,
difference conditions, initial conditions etc.) from \cite{KR-idfind}.

Several important considerations motivated the present search for identities.

Firstly, in \cite{Deb-thesis}, the second author was able to deduce conjectures of Rogers-Ramanujan-type
by analyzing the standard modules for the affine Lie algebra $A_2^{(2)}$ at level 4.
These were found by using Meurman and Primc's variant \cite{MP-ideals}  of Lepowsky and Wilson's $Z$-algebraic
mechanism \cite{LW-Z1,LW-Z2} for finding and proving new identities using standard modules for affine Lie algebras.
The remarkably striking feature of these identities is the appearance of an infinite list of conditions on the sum-sides. 
It remains a worthy (and perhaps a difficult) goal to author an automated search
vastly generalizing \texttt{IdentityFinder} targeted towards such kinds of identities and
understanding how the $Z$-algebraic mechanism (or the variant in \cite{MP-ideals}) works is therefore an
important first step.

In Lepowsky-Wilson's $Z$-algebraic interpretation \cite{LW-Z1,LW-Z2} of the sum-sides in Rogers-Ramanujan
identities, sub-partitions are eliminated (equivalently, forbidden to appear)
based on their lexicographical ordering (see also \cite{MP-ideals}).  
In other words, they treat the ``difference-2-at-distance-1'' condition in the Rogers-Ramanujan 
identities as forbidding the appearance of ``flattest'' 
length $2$-partitions as sub-partitions. The flattest length 2 partition of
$2n$ is $n+n$, while that of $2n+1$ is $(n+1) + n$.
A large subset of conditions on the identities in \cite{Deb-thesis} could be
interpreted this way. 

We are specifically motivated by the following ``affine rank 2''-type, that is, similar to the affine Lie algebras $A_1^{(1)}$ (also known as $\widehat{\mathfrak{sl}_2}$) and $A_2^{(2)}$ situations.
Roughly speaking, a relation with leading term being a 
``square'' of a certain vertex operator 
(as in $\widehat{\mathfrak{sl}_2}$ level 3) corresponds to 
the flattest $2$-partitions being forbidden (for example, in the case of
Rogers-Ramanujan identities).  If one has several such ``generic'' relations with leading term being
a quadratic, then one can eliminate as many first flattest
$2$-partitions.  Sometimes, one has ``non-generic'' quadratic
relations, meaning that the relevant matrices formed by the leading terms are sometimes singular,
which leads to conditional elimination of flattest partitions (for
example, in the case of Capparelli's identities). 
At higher levels, higher powers (as opposed to quadratics) of vertex operators are present as the leading terms in the relations, thereby
engendering the ``difference-at-a-distance''--type conditions. 
Our aim is to capture such phenomena.  Genuinely higher rank situations
give rise to various families of $Z$-operators and therefore to
multi-color partition identities. Note that many seemingly higher rank
algebras at low levels also yield ``affine rank 2''-type conditions.
Also observe that this is  a very crude guideline to the inner mechanics of $Z$-algebras, 
particular situations often involve many quirks.

Secondly, it was a question of Drew Sills if \texttt{IdentityFinder} (possibly with some modifications) is able
to capture identities like the following identity of MacMahon:
\begin{quote}
	Partitions of $n$ with no appearance of consecutive integers as parts and all parts at least $2$
	are equinumerous with partitions of $n$ in which each part is divisible by either $2$ or $3$.
\end{quote}
and a generalization due to Andrews:
\begin{quote}
	Partitions of $n$ with no appearance of consecutive integers as parts,
	no part being repeated thrice and all parts are greater than $1$
	are equinumerous with partitions of $n$ in which each part is $\Mod{2,3,4}{6}$.
\end{quote}
Sum-sides in both of these identities could be easily recast into the ``forbidding-flattest-partitions'' language.
For instance, in the case of MacMahon's identity,
forbidding the appearance of consecutive integers as parts
is equivalent to forbidding the appearance, as  sub-partitions, of flattest 2-partitions of
any odd integer. 
For other examples of such reinterpretations of various known identities, see Section \ref{sec:Known}.
Sequence avoiding partitions as the ones appearing in MacMahon's identity are also of an independent interest,
see for instance, \cite{And-MacMahonBijection,And-sequenceavoiding,HLR-sequenceavoiding,BMN-sequenceavoiding} etc.

It should now be clear to the reader that a framework built on forbidding flattest partitions can unlock a treasure of many new such identities
(and it indeed does, as we report in this paper).

One further observation proved quite useful in searching for these identities. 
In many well-known identities,
the initial conditions are  implied by the difference conditions if one  appends one or more fictitious $0$ parts
to the partitions (for example, MacMahon's identity and its generalization due to Andrews recalled above, 
the second of the Rogers-Ramanujan identities and so on). 
This phenomenon is quite well-known (as was pointed out to us by Drew Sills); see the description
of identities in \cite{And-SecondSchur} for instance.

We present several new families of identities.
Many of the identities reported here have the ``sequence avoiding'' feature as in the MacMahon identity above, and many identities are direct generalizations of MacMahon's identity.
Hence, we loosely chose to call these identities as identities of Rogers-Ramanujan-MacMahon type.
Quite contrary to our expectations, to the best of our knowledge, none of the identities presented here
are principally specialized characters of standard modules for affine Lie algebras at positive integral levels.
Some such identities may lie much deeper in the search space, or perhaps even more innovative searching parameters
are required.  
As a testimony to the former, several ideas of this article along with \cite{K-GG}
helped us identify conjectures related
to certain level 2 modules for $A_9^{(2)}$ which we presented in our article \cite{KR-a92}.
As will be clear from our discussion below, these identities lie quite deep in our current search space and hence had to be found by completely different methods.
Many of the conjectures reported in \cite{KR-a92} now stand proved thanks to the efforts of Bringmann, Jennings-Schaffer and Mahlburg \cite{BJM-a92}.
Lastly, we mention that the present search has a rather broad search space, hence many times an ad-hoc zooming into  the search space was required.

A majority of the identities presented below are proved bijectively, using the works of Xiong and Keith \cite{XK}, Pak and Postnikov \cite{PakPost}, Stockhofe \cite{Sto-bij} and Sylvester \cite{Sylvester-part}. One family of identities is proved by using Appell's theorem.

\subsection*{Future work and work in progress}

Extending the proof technique in Family \refid{fam:originalfamily10}\, is work in progress.

In our search, we worked with a specific ordering on the partitions (explained below). It would be very interesting to search with different orderings.

Several of the identities reported in this article quickly generalize to ``multi-color'' identities. 
We are in the process of significantly generalizing our current search to include these multi-color generalizations.

It will be very interesting to search for identities solely based on the recurrences for sum-side conditions. One advantage is that such a search is fast. This is an ongoing project.

\subsection*{Acknowledgments}

First and foremost, we are extremely grateful to Drew Sills for his question,
his many suggestions and insights which he shared with us throughout this project and for 
his continued encouragement,
all of which were major motivations for the investigations reported in this paper.
We sincerely appreciate the guidance we have received from James Lepowsky, Robert L. Wilson and Doron
Zeilberger over the years. We are thankful to Karl Mahlburg for illuminating discussions regarding bijective proofs.  It is our highest privilege to acknowledge the interest shown by
George E. Andrews in our work.

\section{Some known identities}\label{sec:Known}

Let us first standardize the conventions used in this paper.
Let $\mu = m_1 + \dots + m_r$ and $\pi = p_1 + \cdots + p_s$ be partitions
of $n$. 
If $\pi$ is a partition of $n$, we say that the \emph{weight} of $\pi$ is $n$.
Partitions will always be in \emph{non-increasing} order (however, we 
shall present the identities in a manner independent of order).

By a \emph{$k$-partition} of $n$, we mean a partition of length $k$ of $n$.

We say that $\mu < \pi$ (or that $\mu$ is \emph{flatter} than $\pi$) if
either of the following holds:
\begin{itemize}
	\item $r > s$ (this will not really be needed; we will
	only  compare partitions of same length.)
	\item $r=s$ and $m_1 = p_1, m_2 = p_2, \dots, m_{i-1} = p_{i-1}$ but
	$m_i<p_i$ for some $i$ with $1\leq i\leq r$.
\end{itemize}

\medskip

\begin{exam-non}
	Here are the $4$-partitions of $10$ arranged from flattest to steep,
	i.e., from lexicographically smallest to largest:
	\begin{align*}
	(3, 3, 2, 2)&< (3, 3, 3, 1)< (4, 2, 2, 2)< (4, 3, 2, 1)< (4, 4, 1, 1)\\ &< (5, 2, 2, 1)< (5, 3, 1, 1)< (6, 2, 1, 1)< (7, 1, 1, 1).
	\end{align*}
\end{exam-non}

Following is a (highly non-exhaustive) list of difference conditions in some well-known partition
identities recast in terms of forbidden sub-partitions.
We ignore initial conditions, focusing only on the global difference conditions.

\begin{enumerate}[leftmargin=*]
	\item Rogers-Ramanujan: flattest $2$-partitions are forbidden.
	\item Gordon-Andrews (modulo $2k+1$): flattest $k$-partitions are
	forbidden.
	\item Andrews-Bressoud (modulo $2k$): flattest $k$-partitions are
	forbidden, flattest $k-1$-partition of $n'$ is forbidden if $n'$
	satisfies a specific parity condition $\mod{2}$.
	\item Capparelli: flattest $2$-partitions are forbidden and for all
	$n' \not\equiv 0 \pmod 3$, second flattest $2$-partition of $n'$ is
	forbidden.
	\item Schur: flattest $2$-partitions are forbidden, second flattest
	$2$-partitions of even numbers are forbidden, second flattest
	$2$-partitions of numbers divisible by $3$ are forbidden.  The last
	two conditions can be combined to give: second flattest $2$
	partitions of numbers $\not\equiv \pm 1 \,\,(\mathrm{mod}\,{6})$ are forbidden.
	\item G\"{o}llnitz-Gordon: flattest $2$-partitions are forbidden, second
	flattest $2$- partitions of numbers $\equiv 2\mod{4}$ are forbidden.
	\item MacMahon: 
	flattest $2$-partitions of odd numbers are forbidden.
	\item Andrews: Recall that this identity states that the number of partitions of $n$ into
	parts congruent to $2$, $3$, or $4$, modulo $6$ equals the number of
	partitions of $n$ into parts greater than $1$ where no two consecutive
	integers may appear as parts and a given part may be repeated, but
	not more than twice. Recast: flattest $2$-partitions of odd numbers
	are forbidden, flattest $3$-partitions of numbers divisible by $3$
	are forbidden.  
	\item Symmetric Mod-9s \cite[$I_1, I_2, I_3$]{KR-idfind}: flattest 2-partition of $n'$ if
	$n'\not\equiv 0 \pmod 3$ is forbidden, first two flattest
	$3$-partitions for all $n'$ are forbidden.
	
	\item Identities $1$, $2$, $3$ from \cite{KR-a92}: flattest $2$-partitions of odd numbers forbidden,
	flattest $2$-partitions of $n'$ with $n'\equiv 2 \pmod 4$ forbidden, second flattest $3$-partitions of 
	any $n'$ with $n'\equiv \pm 2 \pmod 6$ forbidden, 
	third and fourth flattest $3$-partition of $n'$ with $n'\equiv 3\pmod 6$ forbidden.
	As one can see, these identities lie deep in our current search space.
\end{enumerate}

\section{The method}
For every $n$, let $\cC(n)$ be a certain subset of partitions of $n$.
We prescribe $\cC$ by imposing flattest-partition
conditions on the partitions.
Let $\cC_j(n)$ be those partitions in $\cC(n)$ with largest
part at most $j$.

Let 
\begin{align*}P(q) &= 1+\sum_{m\geq 1}\left\vert \cC(m) \right\vert q^m, \quad\quad
P_j(q) = 1+\sum_{m\geq 1}\left\vert \cC_j(m)\right\vert q^m.
\end{align*}
be the corresponding 
generating functions. We calculate several coefficients of $P$
(say, up to order $q^{25}$) then employ Euler's algorithm \cite{And-qSeriesBook} to see 
if $P$ has a chance to factor as an interesting (periodic)
infinite product of the form $\prod_{m\geq 1}(1-q^m)^{a_m}$.
If so, we have a potential candidate for an identity.
We use Euler's algorithm as implemented in Garvan's $q$-series maple package
\cite{G-qSeriesPackage}.

To verify a given potential candidate to a high degree of certainty, we proceed as in \cite{KR-idfind}. We first find recursions satisfied by $P_j$. We utilize these recursions to calculate 
$P_N$ up to the order $q^N$ for a large value of $N$.
Finally, we check if $P_N$ also factorizes similarly. Note that 
$$P-P_N\in q^{N+1}\mathbb{N}[[q]].$$
Seldom, these recursions will lead to easy proofs,  for example, in the case of
Identity \refid{id:not3mod4}.

\section{Search space}\label{sec:Searchspace}

The natural search space here is a collection of conditions:
\begin{quote}
	Parameters: $N, A_i,B_i,C_i,D_i,Bool_i$ 
	
	For each $i=1,\dots,N$:
	
	$A_i$th flattest length $B_i$ partition of any $n'$ is forbidden to appear
	as a sub-partition if $n'\equiv C_i\mod{D_i}$.  The boolean bit
	$Bool_i$ toggles between $\equiv$ and $\not\equiv$, 
\end{quote}

Many well-known identities have the following property of initial
conditions:

\begin{quote}
	A partition $\pi$ satisfies the difference conditions and the
	initial conditions \newline if and only if \newline $\pi+0$, i.e., $\pi$
	adjoined with a ``fictitious $0$'' part satisfies the difference
	conditions.
\end{quote}

We utilize this criterion to impose natural initial conditions.
Sometimes, adding more than one fictitious zeros could lead to
interesting identities.

\begin{rmk}Many identities come in pairs or sets (like Rogers-Ramanujan), and in
	such cases, at least one identity in the set seems to satisfy this criterion.
	For the second Capparelli identity, the initial condition that $2$ does not appear 
	could be replaced by assuming a fictitious $-1$ as a part.
\end{rmk}

\begin{rmk}
	Six new conjectural identities were found in \cite{KR-idfind}. It can be 
	checked that the initial conditions in the identities $I_2$--$I_6$ in \cite{KR-idfind}
	are all given by one or more fictitious zeros. $I_1$ does not have an initial condition.
	
	In \cite{MR-thesis} three more identities, called $I_{4a}, I_{5a}, I_{6a}$ were found
	as companions to the corresponding identities in \cite{KR-idfind}. These identities involved
	initial conditions which at first sight seem very mysterious. However, again, it can be
	checked that the initial conditions in $I_{4a}, I_{5a}, I_{6a}$ can be substituted with
	fictitious zero(s). There is a tiny bit of adjustment needed for $I_{6a}$ which we leave to the reader.
	
	One may find more examples of this phenomenon in \cite{KR-a92}, for example, initial conditions in Identity $3$ could be replaced by two fictitious zeros.
\end{rmk}	

\section{Results}

We will express the identities in the following way:

\desc{Product} Condition ``P''

\desc{Sum} Condition ``S''

\desc{Conjugate} Condition ``C''

\desc{Flat form} Condition ``F''

This corresponds to the statement that  for any $n$, partitions 
partitions satisfying condition ``P'' are equinumerous with
partitions satisfying condition ``S'', and moreover,
the generating function for the former class of partitions
can be expressed as a periodic infinite product.

In condition ``C'', we will describe the conditions obtained when the sum-side partitions are replaced by their conjugates (transposing the Ferrer's diagram). 
We will omit the proof of equivalence of Condition ``S'' and Condition ``C''.

In condition ``F'', we will encode the difference conditions on the 
sum-sides in the ``forbidding flattest
partitions'' format using the following convention:

\begin{quote}
	$\forbid{A}{B}{C}{D}$ corresponds to forbidding the appearance, as a
	sub-partition, of the $A$th flattest length $B$ partition of any
	number that is $\Mod{C}{D}$. We may also use $\not\equiv$ as necessary.
\end{quote}

\bigskip

The first few families of identities are either direct generalizations of MacMahon's identity recalled in the Introduction or resemble it closely. We shall provide bijective proofs of these identities.

Generalizations of MacMahon's partition identity were provided by Andrews~\cite{And-MacMahonGen}, and later by Subbarao~\cite{Subbarao-bij}. Then, Andrews, Eriksson, Petrov, and Romik provided a bijective proof of MacMahon's partition identity~\cite{And-MacMahonBijection}. A different bijective proof of MacMahon's partition identity was provided by Fu and Sellers~\cite{FS-MacMahon}, who also extended this new bijection to cover the generalizations of Andrews and of Subbarao, along with a new extension of their own.

\newfamily 
This family is composed of three infinite sub-families. 
\label{fam:originalfamily10}
Fix $k\ge 1$. 

\medskip

\noindent\textbf{Family 1.1.}

\desc{Product}
Parts are either multiples of 3 or congruent to $\pm 2 \pmod{3k+3}$.

\desc{Sum}
\begin{sumside}
	\item Difference between adjacent parts is not $1$.
	\item If the difference between adjacent parts is in $\{2, 5,\dots , 3k -4\}$, then {the smaller of these parts must be $\nMod{2}{3}$}.
	\item If the difference between adjacent parts is in $\{4,7,\dots,3k-2\}$, then 
	{the smaller of these parts must be $\Mod{1}{3}$}.
	\item Initial conditions are given by a fictitious zero, i.e., no parts are equal to $1$, $4$, \dots , $3k- 2$.
\end{sumside}

\desc{Conjugate}
\begin{sumside}
	\item No part appears exactly once.
	\item  If the frequency of a part belongs to $\{2, 5, 8,\dots , 3k- 4\}$, then {the number of parts that are strictly greater than it must be $\nMod{2}{3}$}.
	\item If the frequency of a part belongs to $\{4, 7, 10,\dots, 3k - 2\}$, then the number of parts that are strictly greater than it must be $\Mod{1}{3}$.
\end{sumside}

\medskip

\noindent\textbf{Family 1.2.}

\desc{Product}
Parts are either multiples of 3 or congruent to {$\Mod{-4,-2}{3k+3}$}.

\desc{Sum}
\begin{sumside}
	\item Difference between adjacent parts is not $1$.
	\item If the difference between adjacent parts is in $\{2, 5,\dots , 3k -4\}$, then {the smaller of these parts must be $\nMod{0}{3}$}.
	\item If the difference between adjacent parts is in $\{4,7,\dots,3k-2\}$, then 
	{the smaller of these parts must be $\Mod{2}{3}$}.
	\item Initial conditions are given by a fictitious zero, i.e., no parts are equal to $1$, $4$, \dots , $3k- 2$ or $2, 5,\dots , 3k -4$.
\end{sumside}

\desc{Conjugate}
\begin{sumside}
	\item No part appears exactly once.
	\item  If the frequency of a part belongs to $\{2, 5, 8,\dots , 3k- 4\}$, then the number of parts that are strictly greater than it must be $\nMod{0}{3}$.
	\item If the frequency of a part belongs to $\{4, 7, 10,\dots, 3k - 2\}$, then the number of parts that are strictly greater than it  must be $\Mod{2}{3}$.
\end{sumside}

\medskip

\noindent\textbf{Family 1.3.}

\desc{Product}
Parts are either multiples of 3 or congruent to {$\Mod{2,4}{3k+3}$}.

\desc{Sum}
\begin{sumside}
	\item Difference between adjacent parts is not $1$.
	\item If the difference between adjacent parts is in $\{2, 5,\dots , 3k -4\}$, then {the smaller of these parts must be $\nMod{1}{3}$}.
	\item If the difference between adjacent parts is in $\{4,7,\dots,3k-2\}$, then 
	{the smaller of these parts must be $\Mod{0}{3}$}.
	\item Initial conditions are given by a fictitious zero, i.e., no part is equal to $1$.
\end{sumside}

\desc{Conjugate}
\begin{sumside}
	\item No part appears exactly once.
	\item  If the frequency of a part belongs to $\{2, 5, 8,\dots , 3k- 4\}$, then the number of parts that are strictly greater than it must be $\nMod{1}{3}$.
	\item If the frequency of a part belongs to $\{4, 7, 10,\dots, 3k - 2\}$, then the number of parts that are strictly greater than it  must be $\Mod{0}{3}$.
\end{sumside}

\medskip

We now recall the necessary tools required to prove this family of identities.

Glaisher's Theorem (due to J.\ W.\ L.\ Glaisher~\cite{Gla}), a generalization of Euler's Identity, states that, for fixed modulus $m\ge 2$ and all nonnegative integers $n$, the number of partitions of $n$ with no parts congruent to $0 \pmod m$ equals the number of partitions of $n$ with no part occurring $m$ or more times. A natural question to ask is whether or not there is a bijective proof of Glaisher's Theorem that ``acts'' similarly to Sylvester's bijection. As it turns out, a bijection originally due to D.\ Stockhofe~\cite{Sto-bij} does the trick. Accordingly, X.\ Xiong and W.\ Keith~\cite{XK} provided a refinement of Glaisher's Theorem using a small extension of Stockhofe's bijection. We will give this refinement immediately after defining a few new bits of terminology.

Let the length type of a partition with no parts congruent to $m$ be the $(m-1)$-tuple $\left(\alpha_1,\alpha_2,\dots,\alpha_{m-1}\right)$, where there are $\alpha_i$ parts congruent to $i \pmod m$. Let the alternating sum type of a partition in which no part occurs $m$ or more times be the $(m-1)$-tuple $\left(M_1-M_2,M_2-M_3,\dots,M_{m-1}-M_{m}\right)$, where 
$M_i$ is the sum of all parts in the partition whose index is congruent to $i \pmod m$.

\begin{thm}[\cite{XK}]
	Consider a modulus $m$ and a nonnegative integer $n$. The number of partitions of $n$ with no parts congruent to $0 \pmod m$ and with length type $\left(\alpha_1,\alpha_2,\dots,\alpha_{m-1}\right)$ equals the number of partitions of $n$ with no part occurring $m$ or more times with alternating sum type $\left(\alpha_1,\alpha_2,\dots,\alpha_{m-1}\right).$
\end{thm}
We will not give the details of the bijection here, but direct the reader to the works of Xiong and Keith~\cite{XK} and Stockhofe~\cite{Sto-bij} for more information. 
The reader is invited to verify for herself that, in the case $m=2$, this reduces down to the properties of Sylvester's bijection to be used for the identities below. For our purposes, we will use the case $m=3$ to provide a proof of Family \refid{fam:originalfamily10}\,\, which gives 
a new generalization of MacMahon's identity.

\begin{proof}[Proof of Family \refid{fam:originalfamily10}]
	We shall only prove Family 1.1, the other two families being similar.
	
	Consider a partition of $n$ counted in product side. For the parts that are congruent to $0 \pmod 3$, replace all parts $3j$ with three copies of the part $j$. Set these parts aside for the time being.
	
	Now, consider the parts that are congruent to $\pm 2 \pmod{3k+3}$. Let the number of parts congruent to $2 \pmod{3k+3}$ be $\alpha_1$, and the number of parts congruent to $-2 \pmod{3k+3}$ be $\alpha_2$.
	These parts can be written as either $\left(3k+3\right)m_j-(3k-1)$ or $\left(3k+3\right)m_j-2$, respectively. Map these parts to $3m_j-2$ and $3m_j-1$, respectively.
	This now provides a partition in which no part is a multiple of 3. At this stage, use Stockhofe's bijection to obtain a partition $\mu_1+\mu_2+\mu_3+\cdots$ in which each part appears at most twice which has length type $\left(\alpha_1,\alpha_2\right)$.
	
	Let
	\begin{align*}
	M_1 &= \mu_1 + \mu_4 + \mu_7 + \cdots, \\
	M_2 &= \mu_2 + \mu_5 + \mu_8 + \cdots, \\
	M_3 &= \mu_3 + \mu_6 + \mu_9 + \cdots,
	\end{align*}
	so the alternating sum type is $\left(M_1-M_2,M_2-M_3\right).$
	
	Now:
	\begin{itemize}
		\item Replace each part  $\mu_1,\mu_4,\mu_7,\dots$ with two copies of that part.
		\item Replace each part $\mu_2,\mu_5,\mu_8,\dots$ with $(3k-1)$  copies of that part.
		\item Replace each part $\mu_3,\mu_6,\mu_9,\dots$ with two copies of that part.
	\end{itemize}
	We need to verify that all of these operations restore the partition to its original weight. We just added back in a sum of $M_1 + (3k-2)M_2+M_3$. But, we know
	\begin{align*}
	M_1 - M_2 &= \alpha_1 \\
	M_2 - M_3 &= \alpha_2 \\
	M_1 + M_2 + M_3 &= 3\sum m_j -2\alpha_1 - \alpha_2
	\end{align*}
	Now,
	\begin{align*}
	&M_1 + (3k-2)M_2+M_3 \\ &= 
	k\left(M_1+M_2+M_3\right)
	-\left(k-1\right)\left(M_1-M_2\right)+\left(k-1\right)\left(M_2-M_3\right) \\
	&=k\left(3\sum m_j -2\alpha_1 - \alpha_2\right)-\left(k-1\right)\alpha_1+\left(k-1\right)\alpha_2 \\
	&=\sum 3km_j -\left(3k-1\right)\alpha_1-\alpha_2,
	\end{align*}
	so we are adding $\sum 3km_j -\left(3k-1\right)\alpha_1-\alpha_2$ back into our partition.

	Restoring the ``set aside'' parts that come in triples from the very start of the proof, we have obtained a partition that satisfies the conditions as in the Conjugate formulation.
	
	Now we produce a candidate for the inverse map.
	Consider a partition $\pi$  satisfying the conditions of the conjugate formulation.
	Break $\pi$ into five pieces: 
	\begin{enumerate}
		\item In $\pi_1$ collect those parts of $\pi$ whose frequency is divisible by $3$.
		\item In $\pi_2$ collect those parts that have frequency $\Mod{2}{3}$ such that the number of strictly larger parts is $\nMod{2}{3}$.
		Note that parts of $\pi$ with frequency belonging to $\{2,5,\dots,3k-4	\}$ are exactly the parts accounted in $\pi_2$.
		\item In $\pi_3$ collect those parts that have frequency $\Mod{2}{3}$ such that the number of strictly larger parts is $\Mod{2}{3}$.
		Clearly, any part appearing in $\pi_3$ has frequency at least $3k-1$.
		\item In $\pi_4$ collect those parts that have frequency $\Mod{1}{3}$ such that the number of strictly larger parts is $\Mod{1}{3}$.
		Parts of $\pi$ with frequency belonging to $\{4,7,\dots,3k-2\}$ are exactly the parts accounted in $\pi_4$.
		\item In $\pi_5$ collect those parts that have frequency $\Mod{1}{3}$ such that the number of strictly larger parts is $\nMod{1}{3}$.
		Any part appearing in $\pi_5$ has frequency at least $3k+1$.
	\end{enumerate}

	Retain $2$ copies of each part appearing in $\pi_2$ and move the rest of the copies to $\pi_1$.
	Retain $3k-1$ copies of each part appearing in $\pi_3$ and move the rest of the copies to $\pi_1$.
	Retain $4$ copies of each part appearing in $\pi_3$ and move the rest of the copies to $\pi_1$.
	Retain $3k+1$ copies of each part appearing in $\pi_3$ and move the rest of the copies to $\pi_1$.
	Denote the new $\pi_1$ by $\pi_1'$. 
	After this, keep only $1$ copy of each part appearing in $\pi_2$ and $\pi_3$, discard rest of the copies, and call
	the new partitions $\pi_2'$ and $\pi_3'$. Similarly get $\pi_4'$ and $\pi_5'$ by retaining $2$ copies of each part in $\pi_4$ and
	$\pi_5$ respectively.
	
	Coalesce every tuple of $3$ copies of a part $j$ from $\pi_1'$ into a new part $3j$, and call the new partition $\pi_1''$ and keep this aside.
	
	Consider $\mu=\pi_2'+\pi_3'+\pi_4'+\pi_5'$ and map this via inverse of Stockhofe's bijection we used above to obtain a new partition $\mu'$
	in which no part is a multiple of $3$. In $\mu'$, send every part $3m_j-1$ to the part $\left(3k+3\right)m_j-(3k-1)$ 
	and every part $3m_j-2$ to $\left(3k+3\right)m_j-2$. Call this new partition $\mu''$.
	
	Finally, merge $\pi_1''$ and $\mu''$.

	We leave it to the reader to convince herself that this is indeed the inverse map.

\end{proof}

\newfamilyexample{Example} Letting $k=1$ in Family \refid{fam:originalfamily10}.1 recovers MacMahon's identity (the second and third conditions on the sum-side are vacuous).

Now we present identities obtained with $k=2$ which were the ones found by our computer program.

\newfamilyexample{Example} Take $k=2$ in Family \refid{fam:originalfamily10}.1.

\desc{Product} $\Mod{0,2,3,6,7}{9}$

\desc{Sum}
\begin{sumside}
	\item Difference between adjacent parts is not $1$.
	\item If the difference between adjacent parts is $2$ then their sum is $\nMod{0}{6}$.
	\item If the difference between adjacent parts is $4$ then their sum is $\nMod{2,4}{6}$.
	\item Initial conditions given by a fictitious zero, i.e., smallest part is not $1$ or $4$.
\end{sumside}

\desc{Conjugate}
\begin{sumside}
	\item Difference between adjacent parts is not $1$.
	\item If a part appears exactly twice then the number of parts bigger
	than it is $\nMod{2}{3}$.
	\item If a part appears exactly four times then the number of parts bigger
	than it is $\Mod{1}{3}$.
\end{sumside}

\desc{Flat form} Forbid $\forbid{1}{2}{1}{2}$, $\forbid{2}{2}{0}{6}$, $\forbid{3}{2}{2}{6}$, and $\forbid{3}{2}{4}{6}$.

\begin{proof}[Recursions]
	Even though we have provided a proof above,  we
	also provide the following recursions as they will lead to a nice pattern.
	\begin{align*}
	P_1&= 1, \quad P_2=\dfrac{1}{1-q^2},\quad P_3=P_3=\dfrac{1}{1-q^3} + \dfrac{1}{1-q^2} -1,\\
	P_{3k} &= P_{3k-1} + \dfrac{q^{3k}}{1-q^{3k}}
	\left(P_{3k-2} - P_{3k-4} + P_{3k-5}\right)\\
	P_{3k+1} &= P_{3k} + \dfrac{q^{3k+1}}{1-q^{3k+1}}
	\left(P_{3k-2} - P_{3k-3} + P_{3k-4}\right)\\
	P_{3k+2} &= P_{3k+1} + \dfrac{q^{3k+2}}{1-q^{3k+2}}
	P_{3k}. \qedhere
	\end{align*}
\end{proof}

\newfamilyexample{Example} Take $k=2$ in Family \refid{fam:originalfamily10}.2.

\desc{Product} $\Mod{0,3,5,6,7}{9}$

\desc{Sum}
\begin{sumside}
	\item Difference between adjacent parts is not $1$.
	\item If the difference between adjacent parts is $2$ then their sum is $\nMod{2}{6}$.
	\item If the difference between adjacent parts is $4$ then their sum is $\nMod{0,4}{6}$.
	\item Initial conditions given by a fictitious zero, i.e., smallest part is not $1$, $2$ or $4$.
\end{sumside}

\desc{Conjugate}
\begin{sumside}
	\item Difference between adjacent parts is not $1$.
	\item If a part appears exactly twice then the number of parts bigger
	than it is $\nMod{0}{3}$.
	\item If a part appears exactly four times then the number of parts bigger
	than it is $\Mod{2}{3}$.
\end{sumside}

\desc{Flat form} Forbid $\forbid{1}{2}{1}{2}$, $\forbid{2}{2}{2}{6}$, $\forbid{3}{2}{0}{6}$, and $\forbid{3}{2}{4}{6}$.
\begin{proof}[Recurrences]
	\begin{align*}
	P_1&=P_2= 1, \quad P_3=P_4=\dfrac{1}{1-q^3},\quad P_5=\dfrac{1}{1-q^5} + \dfrac{1}{1-q^3} -1,\\
	P_{3k} &= P_{3k-1} + \dfrac{q^{3k}}{1-q^{3k}}
	P_{3k-2}\\
	P_{3k+1} &= P_{3k} + \dfrac{q^{3k+1}}{1-q^{3k+1}}
	\left(P_{3k-1} - P_{3k-3} + P_{3k-4}\right)\\
	P_{3k+2} &= P_{3k+1} + \dfrac{q^{3k+2}}{1-q^{3k+2}}
	\left(P_{3k-1}-P_{3k-2}+P_{3k-3} \right). \qedhere
	\end{align*}
\end{proof}

\newfamilyexample{Example} Take $k=2$ in Family \refid{fam:originalfamily10}.3.

\desc{Product} $\Mod{0,2,3,4,6}{9}$

\desc{Sum}
\begin{sumside}
	\item Difference between adjacent parts is not $1$.
	\item If the difference between adjacent parts is $2$ then their sum is $\nMod{4}{6}$.
	\item If the difference between adjacent parts is $4$ then their sum is $\nMod{0,2}{6}$.
	\item Initial conditions given by a fictitious zero, i.e., smallest part is not $1$.
\end{sumside}

\desc{Conjugate}
\begin{sumside}
	\item Difference between adjacent parts is not $1$.
	\item If a part appears exactly twice then the number of parts bigger
	than it is $\nMod{1}{3}$.
	\item If a part appears exactly four times then the number of parts bigger
	than it is $\Mod{0}{3}$.
\end{sumside}

\desc{Flat form} Forbid $\forbid{1}{2}{1}{2}$, $\forbid{2}{2}{4}{6}$, $\forbid{3}{2}{0}{6}$, and $\forbid{3}{2}{2}{6}$.
\begin{proof}[Recurrences]
	We have the following recursions.
	\begin{align*}
	P_1&= 1, \,\, P_2=\dfrac{1}{1-q^2},\,\, 
	P_3=\dfrac{1-q^5}{(1-q^3)(1-q^2)},\,\,\\
	P_4&=P_3 + \dfrac{q^4}{(1-q^4)(1-q^2)},\,\,
	P_5=P_4+\dfrac{q^5}{1-q^5}P_3 \\
	P_{3k} &= P_{3k-1} + \dfrac{q^{3k}}{1-q^{3k}}
	\left(P_{3k-3}-P_{3k-4}+P_{3k-5} \right)\\
	P_{3k+1} &= P_{3k} + \dfrac{q^{3k+1}}{1-q^{3k+1}}
	P_{3k-1}\\
	P_{3k+2} &= P_{3k+1} + \dfrac{q^{3k+2}}{1-q^{3k+2}}
	\left(P_{3k}-P_{3k-2}+P_{3k-3} \right).
	\qedhere
	\end{align*}
\end{proof}	

\begin{rmk}
	Note how the recursions in the previous three identities
	are related by a cyclic shift.
\end{rmk}

\newfamily{}
\label{fam:similarfirst}

This is an infinite family, with one identity for every even modulus $\geq 4$.

Fix an even $k\geq 1$.

\desc{Product} Each part is either even or $\Mod{-1}{2k+2}$.

\desc{Sum}
\begin{sumside}
	\item An odd part $2j+1$ is not immediately adjacent to any of the $2j,2j-2,\dots,2j-2k+2$ (its previous $k$ even numbers).
	\item Initial conditions implied by adding a fictitious zero. 
	That is, smallest part is not equal to  $1,3,\dots,2k-1$.
\end{sumside}

\desc{Conjugate} If a part appears exactly $1$, $3$, $\dots$ or $2k-1$ times, then there are an odd number of parts strictly greater than it.

Euler's celebrated partition identity states that, for any nonnegative integer $n$, the number of partitions of $n$ into odd parts equals the number of partitions of $n$ into distinct parts. A key ingredient in our work is the bijective proof of this identity given by J.\ J.\ Sylvester in his classic, colorfully-named treatise on partitions~\cite{Sylvester-part}. This may not be the ``simplest'' proof --- or even the easiest bijective proof --- but it possesses some properties that will be important for us later. (See the work of D.\ Zeilberger~\cite{Zeil-bij} for a recursive formulation of the bijection; for more information on partition bijections, see I.\ Pak's lucid survey article~\cite{Pak-survey}.)

\begin{proof}
	This and the following few families will be proved using Pak and Postnikov's bijection \cite{PakPost}.
	
	Consider a partition of $n$ counted in the product side, i.e., a partition in which each part is either even or $\Mod{-1}{2k+2}$. First, we break all even parts in half: that is, for the parts that are congruent to $0 \pmod 2$, replace all parts $2j$ with two copies of the part $j$.
	
	The remaining parts are all of the form $(2k+2)m_j-1$ for some positive integers $m_j$. Replace each of these parts with $2m_j-1$; we are now considering a partition into odd parts. We send this to a partition into distinct parts \cite{PakPost}. For this new partition $\mu_1+\mu_2+\mu_3+\mu_4+\dots$
	we replace each odd-indexed part with $2k+1$ copies of that part.
	
	The proof that this procedure gets us a partition of correct weight and that 
	this map is a bijection between partitions counted in the product side and the ones counted in the Conjugate formulation is exactly as in the proof of Family \refid{fam:originalfamily10}\, given above.
	
	Now we produce a candidate for the inverse map. 
	Consider a partition $\pi$ of weight $n$ satisfying the conditions of the conjugate formulation. Break $\pi$ into three classes. Collect in $\pi_1$ those parts that appear with an even frequency, collect in $\pi_2$ those parts that appear with  an odd frequency and such that the number of parts that are strictly larger is also odd, and collect in $\pi_3$ those parts that appear with  an odd frequency and such that the number of parts that are strictly larger is even.
	Note that $\pi_2$ necessarily contains all those parts of $\pi$ that appear with an odd frequency $\leq 2k-1$, and any part appearing in $\pi_3$ has frequency at least $2k+1$.
	
	Now, retain one copy of each part appearing in $\pi_2$, and move the rest of the copies to $\pi_1$.
	Retain $2k+1$ copies of each  part appearing in $\pi_3$ and move the rest of the copies (of which there are an even number) to $\pi_1$. After this, only retain a single copy of each part appearing in $\pi_3$ and discard the rest of the copies.
	Call the new partitions $\pi_1'$, $\pi_2'$ and $\pi_3'$. 
	
	For $\pi_1'$, merge two copies of each part $j$ into a new part $2j$, call the new partition $\pi_1''$ and keep it aside.
	
	Consider $\mu=\pi_2'+\pi_3'$. It is not hard to see that in $\mu$ has distinct parts and parts of odd index are precisely the parts coming from $\pi_3'$. Now map $\mu$ to a partition with odd parts $\mu'$. In $\mu'$ map every odd part $2m_j-1$ to $(2k+2)m_j-1$ to obtain a new partition $\mu''$. Finally merge $\pi_1''$ and $\mu''$.
	
	We leave it to the reader to convince herself that this is indeed the inverse map.

\end{proof}

Let us consider a specific example of this. Consider the theorem in the case that $k=2$. The product side allows parts congruent to $0\pmod 2$ and $5 \pmod 6$; for example, consider
$$40+23+14+14+12+11+6+6+6+5+5.$$
First, we replace all of the even parts $2j$ with two copies of $j$, obtaining $$40+14+14+12+6+6+6 \mapsto 20+20+7+7+7+7+6+6+3+3+3+3+3+3.$$
Now consider the remaining odd parts, which are all congruent to $5 \pmod 6$:
$$23+11+5+5.$$
Sending each part of the form $6m_j-1$ to $2m_j-1$ produces 
$$7+3+1+1.$$
This maps to the following partition with distinct parts:
$$7+4+1.$$
We now replace each odd-indexed part with 5 copies of itself, producing
$$7+7+7+7+7+4+1+1+1+1+1.$$
Now, combining this with the previously obtained parts, we finally get
$$20+20+7+7+7+7+7+7+7+7+7+6+6+4+3+3+3+3+3+3+1+1+1+1+1.$$

For the inverse map, check that 
$\pi_1=20+20+6+6+3+3+3+3+3+3$, 
$\pi_2=4$,
$\pi_3=7+7+7+7+7+7+7+7+7+1+1+1+1+1$.

We get $\pi_1'=20+20+7+7+7+7+6+6+3+3+3+3+3+3$, $\pi_2'=4$, $\pi_3'=7+1$.
We have $\pi_2''=40+14+14+12+6+6+6$.
We also have $\mu=7+4+1$, $\mu'=7+3+1+1$ and $\mu''=23+11+5+5$.

Finally, we have $\pi''+\mu''=40+23+14+14+12+11+6+6+6+5+5$.

\begin{rmk}
	All of the families from here until Family \refid{fam:similarlast}\,\,will use a very similar procedure to obtain the bijections. We shall only indicate how the proofs differ, leaving the details to the reader.
\end{rmk}

\newfamilyexample{Example}

\desc{Product} $\Mod{0,2,3}{4}$

\desc{Sum}
\begin{sumside}
	\item An odd part $2j+1$ is not immediately adjacent to $2j$.
	\item Smallest part is not $1$.
\end{sumside}

\desc{Flat form} Forbid $\forbid{1}{2}{1}{4}$.

\begin{proof} This particular identity can be proved quickly using recursions.
	\begin{align*}
	P_{1}&=1,\\
	P_{2j+1}&=\dfrac{q^{2j+1}}{1-q^{2j+1}}P_{2j-1}+
	\dfrac{1}{1-q^{2j}}P_{2j-1}
	=\dfrac{1-q^{4j+1}}{(1-q^{2j})(1-q^{2j+1})}P_{2j-1}.
	\end{align*}
	Now take the limit as $j\rightarrow\infty$.
\end{proof}

\newfamilyexample{Example}

\desc{Product} $\Mod{0,2,4,5}{6}$.     

\desc{Sum}
\begin{sumside} 
	\item An odd part $2j+1$ is not immediately adjacent to
	either of $2j$ or $2j-2$. 
	\item Initial conditions implied by a fictitious zero, i.e., smallest part is not equal to $1$ or $3$.
\end{sumside}

\desc{Flat-form} Forbid $\forbid{1}{2}{1}{4}$, and $\forbid{2}{2}{3}{4}$.

\newfamily{} 
This is an infinite family with one identity for each even modulus $\geq 4$.

Fix a $k\geq 1$.

\desc{Product} Each part is either even or $\Mod{1}{2k+2}$

\desc{Sum}
An even part $2j$ is forbidden to be 
adjacent to either of $2j-1,2j-3,\dots,2j-2k+1$ (its previous $k$ odd numbers).

\desc{Conjugate} 
If a part appears exactly $1$, $3$, $\dots$ or $2k-1$ times, then there are an even number of parts  strictly greater than it.

\begin{proof}
	Consider a partition of $n$ counted in the product side. Then, break all even parts in half. Now, the remaining parts are all of the form $(2k+2)m_j+1$ for some positive integers $m_j$. Replace each of these parts with $2m_j+1$; we are now considering a partition into odd parts. Now, send this to a partition into distinct parts. For this new partition $\mu_1+\mu_2+\mu_3+\mu_4+\dots$
	replace each {even}-indexed part with {$2k+1$} copies of that part.
\end{proof}

\newfamilyexample{Example} \label{id:not3mod4}

\desc{Product} Parts are $\Mod{0,1,2}{4}$.

\desc{Sum} 
An even part $2j$ is forbidden to be immediately adjacent to 
$2j-1$.

\desc{Flat form} Forbid $\forbid{1}{2}{3}{4}$.
\begin{proof}
	Let $P_j$ be the generating function of sum sides, with added restriction that
	largest part is $\leq j$.  Then:
	\begin{align*}
	P_{2j} &= \dfrac{1}{1-q^{2j}}P_{2j-2} - P_{2j-2} + P_{2j-1},\quad 
	P_{2j-1} = \dfrac{1}{1-q^{2j-1}}P_{2j-2}.
	\end{align*}
	Combining, we get:
	\begin{align*}
	P_{2j} &= P_{2j-2}\left(\dfrac{1}{1-q^{2j}}
	+\dfrac{1}{1-q^{2j-1}}-1 \right)= P_{2j-2}\left(\dfrac{1-q^{4j-1}}
	{(1-q^{2j})(1-q^{2j-1})}\right),
	\end{align*}
	Now use $P_2 = \dfrac{1-q^3}{(1-q)(1-q^2)}$ and induct. 
\end{proof}

\newfamilyexample{Example}

\desc{Product} Parts are $\nMod{0,1,2,4}{6}$.

\desc{Sum} 
\begin{enumerate}[noitemsep,topsep=0pt,label=-,leftmargin=2em]
	\item An even part $2j$ is forbidden to be immediately adjacent to 
	$2j-1$ or $2j-3$.
\end{enumerate}

\desc{Flat form} Forbid $\forbid{1}{2}{3}{4}$,  and $\forbid{2}{2}{1}{4}$.

\allowdisplaybreaks

\newfamily{} This is an infinite family, with one identity
for every even modulus greater than or $8$.
For modulus $6$, one of the conditions becomes redundant and
one gets MacMahon's identity recalled in the Introduction.

Fix a $k\geq 1$.

\desc{Product} Parts are either even or $\Mod{3}{2k+6}$.

\desc{Sum}
\begin{sumside}
	\item Difference between adjacent parts is not 1.
	\item An even part $2j$ is not immediately adjacent to
	any of 
	$2j-3,\dots,2j-2k-1$.
	\item Initial condition is implied by a fictitious zero. That is, smallest part is not $1$. 
\end{sumside}	

\desc{Conjugate} 
\begin{sumside}
	\item No part appears exactly once.
	\item If a part appears exactly $3$, $5$, $\dots$, or $2k+1$ times then there are an even number of parts  strictly greater than it.
\end{sumside}

\begin{proof}
	Consider a partition of $n$ counted in the product side. Then, break all even parts in half. Now, the remaining parts are all of the form $(2k+6)m_j+3$ for some positive integers $m_j$. Replace each of these parts with $2m_j+1$; we are now considering a partition into odd parts. Now, send this to a partition into distinct parts. For this new partition $\mu_1+\mu_2+\mu_3+\mu_4+\dots$
	replace each {odd}-indexed part with {$3$} copies of that part
	and each {even}-indexed part with {$2k+3$} copies of that part.
\end{proof}

\newfamilyexample{Example}

\desc{Product} $\Mod{0,2,3,4,6}{8}$.

\desc{Sum}
\begin{sumside}
	\item Difference between adjacent parts is not 1.
	\item An even part $2j$ is not immediately adjacent to
	$2j-3$.
	\item Initial condition is implied by a fictitious zero. That is, smallest part is not $1$. 
\end{sumside}	

\desc{Flat form} Forbid $\forbid{1}{2}{1}{2}$, and $\forbid{2}{2}{1}{4}$.

\newfamily{} This is again an infinite family, with one identity
for every even modulus $\geq 8$. Again, for modulus $6$, one of the conditions becomes redundant and we get MacMahon's identity.

Fix a $k\geq 1$.

\desc{Product} Parts are either even or $\Mod{-3}{2k+6}$.

\desc{Sum}
\begin{sumside}
	\item Difference between adjacent parts is not 1.
	\item An odd part $2j+1$ is not allowed to be immediately adjacent to
	$2j-2,\dots,2j-2k$.
	\item Initial conditions implied by a fictitious zero,
	i.e., smallest part is not equal to $1,3,\dots,2k+1$.
\end{sumside}	

\desc{Conjugate}
\begin{sumside}
	\item No part appears exactly once.
	\item If a part appears exactly $3$, $5$, $\dots$, or $2k+1$ times then there are an odd number of parts  strictly greater than it.
\end{sumside}

\begin{proof}
	Consider a partition of $n$ counted in the product side. Then, break all even parts in half. Now, the remaining parts are all of the form $(2k+6)m_j-3$ for some positive integers $m_j$. Replace each of these parts with $2m_j-1$; we are now considering a partition into odd parts. Now, send this to a partition into distinct parts. For this new partition $\mu_1+\mu_2+\mu_3+\mu_4+\dots$
	replace each {even}-indexed part with {$3$} copies of that part
	and each {odd}-indexed part with {$2k+3$} copies of that part.
\end{proof}

\newfamilyexample{Example}

\desc{Product} Parts are $\Mod{0,2,4,5,6}{8}$.

\desc{Sum}
\begin{sumside}
	\item Difference between adjacent parts is not 1.
	\item An odd part $2j+1$ is not allowed to be immediately adjacent to 
	$2j-2$.
	\item Smallest part is not equal to $1$ or $3$.
\end{sumside}	

\desc{Flat form} Forbid $\forbid{1}{2}{1}{2}$, and  $\forbid{2}{2}{3}{4}$.

\newfamily{} An infinite family with one identity for every
modulus divisible by $4$ and $\geq 12$.

Fix a $k\geq 1$.

\desc{Product} Parts are even or $\Mod{2k+5}{4k+8}$.

\desc{Sum}
\begin{sumside}
	\item Difference between consecutive parts can't be $1,3,\dots,2k+1$.
	\item An odd part $2j+1$ can't be immediately adjacent to
	$2j - 2k-2$.
	\item Initial conditions implied by a fictitious zero,
	that is, the smallest can't be either of $1,3,\dots,2k+3$.
\end{sumside}	

\desc{Conjugate}
\begin{sumside}
	\item No part appears exactly $1$, $3$, $\dots$, $2k+1$ times.
	\item If a part appears exactly $2k+3$ times then there are an odd number of parts strictly  greater than it.
\end{sumside}	

\begin{proof}
	Consider a partition of $n$ counted in the product side. Then, break all even parts in half. Now, the remaining parts are all of the form $(4k+8)m_j+(2k+5)$ for some positive integers $m_j$. Replace each of these parts with $2m_j+1$; we are now considering a partition into odd parts. Now, send this to a partition into distinct parts. For this new partition $\mu_1+\mu_2+\mu_3+\mu_4+\dots$
	replace each {odd}-indexed part with {$2k+5$} copies of that part
	and
	replace each {even}-indexed part with {$2k+3$} copies of that part.
\end{proof}

\newfamilyexample{Example}

\desc{Product} $\Mod{0,2,4,6,7,8,10}{12}$ 

\desc{Sum}
\begin{sumside}
	\item Difference between consecutive parts can't be $1$ or $3$.
	\item An odd part $2j+1$ can't be immediately adjacent to
	$2j -4$.
	\item Initial conditions implied by a fictitious zero,
	that is, the smallest part can't be either of $1,3,5$.
\end{sumside}	

\desc{Flat form} Forbid $\forbid{1}{2}{1}{2}$,
$\forbid{2}{2}{1}{2}$, and $\forbid{3}{2}{1}{4}$.

\newfamily{}\label{fam:similarlast}
An infinite family, with one identity for every
modulus divisible by $4$ that is $\geq 12$.

Fix $k\geq 1$.

\desc{Product} Each part is either even or $\Mod{2k+3}{4k+8}$.

\desc{Sum}
\begin{sumside}
	\item Difference between consecutive parts can't be $1,3,\dots,2k+1$.
	\item An even part $2j$ can't be immediately adjacent to
	$2j - 2k -3$.
	\item Initial conditions are given by a fictitious zero,
	that is, the smallest part is not amongst $1,3,\dots, 2k+1$.
\end{sumside}	

\desc{Conjugate}
\begin{sumside}
	\item No part appears exactly $1$, $3$, $\dots$, $2k+1$ times.
	\item If a part appears exactly $2k+3$ times then there are an even number of parts strictly  greater than it.
\end{sumside}	

\begin{proof}
	Consider a partition of $n$ counted in the product side. Then, break all even parts in half. Now, the remaining parts are all of the form $(4k+8)m_j+(2k+3)$ for some positive integers $m_j$. Replace each of these parts with $2m_j+1$; we are now considering a partition into odd parts. Now, send this to a partition into distinct parts. For this new partition $\mu_1+\mu_2+\mu_3+\mu_4+\dots$
	replace each {even}-indexed part with {$2k+5$} copies of that part
	and
	replace each {odd}-indexed part with {$2k+3$} copies of that part.
\end{proof}

\newfamilyexample{Example}

\desc{Product} Each part is either even or $\Mod{5}{12}$.

\desc{Sum}
\begin{sumside}
	\item Difference between consecutive parts can't be $1$ or $3$.
	\item An even part $2j$ can't be immediately adjacent to
	$2j - 5$.
	\item Smallest part is not $1$ or $3$.
\end{sumside}	

\desc{Flat form} Forbid $\forbid{1}{2}{1}{2}, \forbid{2}{2}{1}{2}$, and  $\forbid{3}{2}{3}{4}$.

\begin{rmk}
	Families \refid{fam:similarfirst}--\refid{fam:similarlast}\,\,are of a very similar nature. It seems very likely that they can all be incorporated into a grand family and proved together. We leave this for an interested reader.
\end{rmk}

\newfamily \label{fam:arxiv} Let $k\geq 2$. \footnote{This family is incorporated from the article \href{https://arxiv.org/abs/1703.04715}{\texttt{arXiv:\ 1703.04715 [math.CO]}}.}

\desc{Product} Each part is either even but $\nMod{2}{4k}$ or odd and $\Mod{1,2k+1}{4k}$

\desc{Sum}
\begin{sumside}
	\item If an odd part $2j+1$ is present, then none of the other parts are equal to any of $2j+1,2j+2,\dots,2j+2k-1$.
\end{sumside}

\medskip

Actually, this family is in a sense dual to the following family of identities due to Andrews, some special cases of which were found by our computer program:

\begin{thm}[Thm.\ 3 \cite{And-SecondSchur}]\label{thm:And1}
	Let $k\geq 2$. 
	\desc{Product} Each part is either even but $\nMod{4k-2}{4k}$ or odd and $\Mod{2k-1,4k-1}{4k}$.
	
	\desc{Sum}
	\begin{sumside}
		\item If an odd part $2j+1$ is present, then none of the other parts are equal to any of $2j+1,2j,\dots,2j-2k+3$.
		\item Smallest part is not equal to any of $1,3,\dots,2k-3$.
	\end{sumside}	
\end{thm}

Over the past decade or so, there has been a lot of interest in exploring overpartition analogues of classical partition identities, 
(as a small and by no means exhaustive sample, see papers by Chen et al., Corteel, Lovejoy, and Dousse \cite{CoLo-over,Lo-RRover,CSS-brover,Dousse-Schurover}).
Overpartitions are partitions in which last occurrence of any part may appear overlined. 

The fact that odd parts are not allowed to be repeated (though even parts may be repeated arbitrarily many times) in the identities above suggests that both are actually special cases of an overpartition theorem. We now present an overpartition generalization that can be used to recover Family \refid{fam:arxiv}\,\, and Theorem \ref{thm:And1}  upon appropriate specializations.

\begin{thm}\label{thm:mainoverthm}
	For $k \ge 2$,  let $A_k(m,n)$ be the number of overpartitions of $n$ with exactly $m$ overlined parts, subject to the following conditions:
	\begin{itemize}
		\item If an overlined part $\overline{b}$ appears then 
		all of the non-overlined parts $b, b+1,\dots, b+k-2$ are forbidden to appear.
		\item If an overlined part $\overline{b}$ appears then 
		all of the overlined parts $\overline{b+1}$, $\overline{b+2}$, $\dots$, $\overline{b+k-1}$ are forbidden to appear.
	\end{itemize}
	Then, 
	\[\sum_{m,n \ge 0} A_k(m,n)a^m q^n = \dfrac{\left(-aq;q^k\right)_{\infty}}{(q;q)_{\infty}}.\]
\end{thm}
\begin{proof}
	Fix $k \ge 2$. Let $p_j(m, n)$ be the number of overpartitions of $n$ 
	with $m$ overlined parts
	that satisfy the conditions in Theorem \ref{thm:mainoverthm}, with the further restriction that all parts are $\leq j$. Let $r_j(m,n)$ be the number of overpartitions of $n$ counted by $p_j(m,n)$ where $\overline{j}$, $\overline{j-1}$, \dots, $\overline{j-k+2}$ do not appear (that is, the largest possible overlined part is $\overline{j-k+1}$). Then, let 
	\begin{align*}
	P_j(a,q) &= \sum_{m,n \ge 0} p_j(m,n)a^mq^n,\\
	R_j(a,q) &= \sum_{m,n \ge 0} r_j(m,n)a^mq^n,
	\end{align*}
	we let $P_0=R_0=1$.
	It is clear that
	\[R_{\infty}(a,q) = P_{\infty}(a,q) = \sum_{m,n \ge 0} D_k(m,n)a^m q^n.\]
	
	Let 
	\[F(a,x,q) = \sum_{j\geq 0}R_j(a,q)x^j.\]
	Observe that the following recursion and initial conditions are satisfied:
	\begin{align*}
	R_j(a,q) &= \dfrac{1}{1-q^j}R_{j-1}(a,q) + \dfrac{aq^{j-k+1}}{1-q^j} R_{j-k}(a,q),\,\, j\geq k.\\
	R_j(a,q) &= \dfrac{1}{(q;q)_j},\,\, 0\leq j < k.
	\end{align*}
	Note the following alternate way to write the recursion and the initial conditions:
	\begin{align*}
	R_j(a,q) &= \dfrac{1}{1-q^j}R_{j-1}(a,q) + \dfrac{aq^{j-k+1}}{1-q^j} R_{j-k}(a,q),\,\, j\geq 1.\\
	R_0(a,q) &= 1,\quad R_{j}(a,q)=0\,\,\text{for } -k<j<0,
	\end{align*}
	which immediately gets us to
	\[
	(1-x)F(a,x,q)=F(a,xq,q)+ax^kqF(a,xq,q) = (1+ax^kq)F(a,xq,q).
	\]
	Noting that 
	\[
	\lim\limits_{n\rightarrow\infty}F(a,xq^n,q)=R_0(a,q)=1,
	\]
	we obtain
	\[F(a,x,q)= \prod\limits_{j\geq 0}\dfrac{1+ax^kq^{jk+1}}{1-xq^j}.\]
	Finally, by Appell's comparison theorem \cite[page 101]{D-taylor} we have: 
	\[
	R_{\infty}(a,q)=
	\lim\limits_{x\rightarrow 1}((1-x)F(a,x,q))
	= \lim\limits_{x\rightarrow 1}
	\left(\prod\limits_{j\geq 0}\dfrac{1+ax^kq^{jk+1}}{1-xq^{j+1}}\right)
	=
	\dfrac{\left(-aq;q^k\right)_{\infty}}{(q;q)_{\infty}}.
	\]
\end{proof}

Now, Family \refid{fam:arxiv}\,\,and Theorem \ref{thm:And1}  can be recovered by appropriate specializations. Letting $(a,q)\mapsto \left(q^{-1},q^2\right)$ (that is, we map every nonoverlined part $j\mapsto 2j$ and every overlined part $\overline{j}\mapsto 2j-1$) gets us Family \refid{fam:arxiv}, while using $(a,q)\mapsto \left(q^{2k-3},q^2\right)$ (now mapping $j\mapsto 2j$ and every overlined part $\overline{j}\mapsto 2j+2k-3$) provides us with Theorem \ref{thm:And1}.

However, many more corollaries can be found. For $k\ge 2$, by choosing $i \in \left\{0,\dots,k-1\right\}$ and letting $(a,q)\mapsto \left(q^{2i-1},q^2\right)$, we get:
\begin{cor}
	Let $B(n)$ be the number of partitions of a non-negative integer $n$ in which each part is either even but $\nMod{4i+2}{4k}$ or odd and $\Mod{2i+1,2k+2i+1}{4k}$. Also, let $C(n)$ be the number of partitions of $n$ in which if an odd part $2j+1$ is present, 
	then none of the other even parts are equal to any of 
	$2j-2i+2, 2j-2i+4, \dots, 2j+2k-2i-2$,
	none of the other odd parts are equal to any of
	$2j+1, 2j+3, \dots, 2j+2k-1$
	and the smallest odd part is at least $2i+1$. 
	Then, $B(n) = C(n)$ for all $n$.
\end{cor}

We leave it to the reader to work out identities related to the specializations
$q\mapsto q^t$ for $t>2$.

\newfamily{} 
\label{fam:easy}
This is perhaps the easiest of the families that we came across, but it is interesting nonetheless.
The proofs of all of the identities in this family follow the same pattern as in Family \refid{id:family1proofexample}{} 
below.

For every modulus, we have $k$ identities that avoid
exactly one congruence class in their product.
This can be greatly generalized. We first start with the ``base'' case.

Fix $k\geq 4$ and let $1\leq j \leq k$.

\desc{Product} Parts are $\nMod{j}{k}$.

\desc{Sum}
\begin{sumside}
	\item If difference at distance $\left\lceil k/2 \right\rceil-1$ is strictly less than $2$, 
	then the sum of these $\left\lceil k/2 \right\rceil$ parts is $\nMod{j}{k}$.
	\item Initial conditions are implied by adding $\left\lceil k/2\right\rceil -1$ fictitious zeros. 
\end{sumside}	

\desc{Flat form} Forbid $\forbid{1}{\left\lceil k/2\right\rceil}{j}{k}$.

\newfamilyexample{Family} Here we present the full set of identities for $k=5$.

\begin{enumerate}[topsep=0pt]    
	\item
	\desc{Product} Parts are $\nMod{1}{5}$.
	
	\desc{Sum}
	\begin{sumside}
		\item If difference at distance 2 is $0$ or $1$ then the sum of these three parts is $\nMod{1}{5}$
		\item Smallest part is at least 2.
	\end{sumside}	
	
	\desc{Flat form} Forbid $\forbid{1}{3}{1}{5}$.

	\item
	\desc{Product}  Parts are $\nMod{2}{5}$.
	
	\desc{Sum}
	\begin{sumside}
		\item If difference at distance 2 is $0$ or $1$ then the sum of these three parts is $\nMod{2}{5}$
		\item $1$ appears at most once.
	\end{sumside}	
	
	\desc{Flat form} Forbid $\forbid{1}{3}{2}{5}$.
	
	\item
	\desc{Product} Parts are  $\nMod{3}{5}$.
	
	\desc{Sum}
	\begin{sumside}
		\item If difference at distance 2 is $0$ or $1$ then the sum of these three parts is $\nMod{3}{5}$
	\end{sumside}	
	
	\desc{Flat form} Forbid $\forbid{1}{3}{3}{5}$.
	
	\item
	\desc{Product} Parts are  $\nMod{4}{5}$.
	
	\desc{Sum}
	\begin{sumside}
		\item If difference at distance 2 is $0$ or $1$ then the sum of these three parts is $\nMod{4}{5}$
	\end{sumside}	
	
	\desc{Flat form} Forbid $\forbid{1}{3}{4}{5}$.
	
	\item
	\desc{Product} Parts are  $\nMod{5}{5}$.
	
	\desc{Sum}
	\begin{sumside}
		\item If difference at distance 2 is $0$ or $1$ then the sum of these three parts is $\nMod{5}{5}$
	\end{sumside}	
	
	\desc{Flat form} Forbid $\forbid{1}{3}{5}{5}$.
	
\end{enumerate}    

\subsection*{Avoiding more number of congruence classes}

This can be generalized to avoiding two or more congruence classes.
The main idea for that is as follows. 
Let $N$ be the intended modulus, and let $S$ be a set of congruence classes we wish to avoid on the product.
Suppose that we are looking for an identity with the following form:

\desc{Product} Parts $\nMod{S}{N}$

\desc{Sum} 
\begin{sumside}
	\item If difference at distance $2$ is strictly less than 2, i.e. ($\lambda_i -\lambda_{i+2}\leq 1$) 
	then the sum of these parts is $\nMod{S}{N}$, i.e., $\lambda_i + \lambda_{i+1} +\lambda_{i+2} \nMod{S}{N}$.
	\item Possibly add fictitious zeros as appropriate.
\end{sumside}

Then, it appears to us that this can always be done as long as elements of $S$ are \emph{sufficiently spread out}.
Below, we show how to do this for a few moduli $N$ and a corresponding sets $S$.

\subsubsection*{Avoiding 2 congruence classes}
We wish to let $S = \{ i,j \}$ be a pair of integers which will be forbidden residues
and let $N$ be a modulus. 

We sketch a proof of some mod$-9$s; the proofs of others are similar.

\newfamilyexample{Family} \label{id:family1proofexample}
A family of Mod$-9$s, $N=9$.

The set $S$ can be taken to be one of 
\[
\{0,3\}, \{0,4\},\{0,5\},\{0,6\}, \quad
\{1,6\}, \quad
\{2,6\},\quad
\{3,6\}, \{3,7\},\{3,8\}.
\]
with no fictitious zeros added.

\begin{proof} We show how this works for a few pairs.
	The pairs $\{0,3\},\{0,6\},\{3,6\}$ yield very easy identities. The rest are very mildly challenging.
	The proofs are similar to Identity \refid{id:not3mod4}.
	
	For $\{0,4\}$, observe that:
	\begin{align*}
	&P_{3n+3}\\&  = \left(1+q^{3n+3} + q^{3n+3}q^{3n+3}\right)\left(\dfrac{q^{3n+2}}{1-q^{3n+2}}(1+q^{3n+1}) + \dfrac{1}{1-q^{3n+1}} \right)P_{3n}\\
	&= \dfrac{(1 - q^{9n+9})(1-q^{9n+4})}{(1-q^{3n+3})(1-q^{3n+2})(1-q^{3n+1})} P_{3n}.
	\end{align*}
	Substituting $P_0=1$ and letting $n\rightarrow\infty$ we get the result.
	
	For $\{0,5\}$, the recursion changes to:
	\begin{align*}
	& P_{3n+3} \\ & = \left(1+q^{3n+3} + q^{3n+3}q^{3n+3}\right)\left(\dfrac{q^{3n+2}q^{3n+2}}{1-q^{3n+2}}  + \dfrac{1}{1-q^{3n+1}}(1+q^{3n+2}) \right)P_{3n}.
	\end{align*}
	
	For $\{3,7\}$, the recursion changes to:
	\begin{align*}
	& P_{3n+4} \\ & = \left(1+q^{3n+4} + q^{3n+4}q^{3n+4}\right)\left(\dfrac{q^{3n+3}}{1-q^{3n+3}}(1+q^{3n+2}) + \dfrac{1}{1-q^{3n+2}} \right)P_{3n+1}.
	\end{align*}
	
	For $\{3,8\}$, the recursion changes to:
	\begin{align*}
	& P_{3n+4} \\ & = \left(1+q^{3n+4} + q^{3n+4}q^{3n+4}\right)\left(\dfrac{q^{3n+3}q^{3n+3}}{1-q^{3n+3}}  + \dfrac{1}{1-q^{3n+2}}(1+q^{3n+3}) \right)P_{3n+1}.
	\end{align*}
	For $\{3,7\}$ and $\{3,8\}$, we let $P_1 = (1-q^3)/(1-q)$.
	Note the similarity of recursions of $\{0,4\}$ with $\{3,7\}$ and $\{0,5\}$ with $\{3,8\}$.
\end{proof}
\newfamilyexample{Family} A family of Mod$-10$s, $N=10$.
The set $S$ takes the values:
\[
\{0,5\}, \quad
\{3,8\}, \{3,9\}, \quad \{4,9\}, \quad \{5,9\} 
\]
with no fictitious zeros added.

Note that $\{0,5\},\{3,8\},\{4,9\}$ are from
an already discovered family of Mod$-5$s.

\newfamilyexample{Family} A family of Mod$-11$s, $N=11$. The set $S$ takes values:
\begin{align*}
\text{With no fictitious zeros:}\quad
& \{0,4\},\{0,5\}, \{0,6\}, \,\,
\{3,7\}, \{3,8\}, \{3,9\}, \\ 
&  \{4,9\}, 
\{4,10\}, \,\,
\{5,10\}, \,\,
\{6,10\}.\\
\text{With 1 fictitious zero:}\quad
& \{2,7\}, \{2,8\}, \{2,9\}.\\
\text{With 2 fictitious zeros:}.\quad
& \{1,6\}, \{1,7\}.
\end{align*}

\subsubsection*{Avoiding 3 congruence classes}
We exhibit this with an example.
\newfamilyexample{Identity}
We continue working with difference at distance $2$ and no fictitious zeroes. 
The set $S=\{0,9,16\}$ of three elements modulo $23$ works.

One may find other pairs $N,S$.

\subsection*{Beyond difference-at-distance 2}
This idea naturally generalizes to conditions with higher distances, as we show with an example.

\newfamilyexample{Identity} 

Let $S=\{0,7\}$, $N=17$.

\desc{Product} Parts are $\nMod{S}{N}$

\desc{Sum}
\begin{sumside}
	\item If difference at distance 3 is strictly less than 2, that is $\lambda_i - \lambda_{i+3}\leq 1$,
	then the sum of these parts, that is, $\lambda_i+\lambda_{i+1}+\lambda_{i+2}+\lambda_{i+3}\nMod{S}{N}$.
\end{sumside}

And so on for other pairs $N, S$ and with conditions at larger distances\dots.

\begin{rmk}
	We leave to the interested reader to work out a precise theorem that covers all of these examples.
\end{rmk}

\bibliographystyle{abbrv}

\onehalfspacing

\end{document}